\documentclass[11pt]{article}
\usepackage{amssymb,amsmath,setspace,graphicx,multicol,wrapfig,amsfonts,amsthm,titling,url,array,parskip,enumitem,color,bbm}
\usepackage{hyperref}
\usepackage{placeins}
\usepackage{aurical,pbsi}
\usepackage[T1]{fontenc}
\usepackage[hmargin=2.5cm,vmargin=2.5cm]{geometry}

\numberwithin{equation}{section}

\newtheoremstyle{slplain}
  {3pt} 
  {3pt} 
  {\slshape}
  {}
  {\bfseries}
  {.}
  { }
  {}

\theoremstyle{slplain}
\newtheorem{thm}{Theorem}[section]
\newtheorem{lem}[thm]{Lemma}

\newtheorem{cor}[thm]{Corollary}

\theoremstyle{slplain}
\newtheorem{defn}{Definition}[section]

\newtheorem*{example}{Example}

\theoremstyle{remark}
\newtheorem*{remark}{Remark}

\makeatletter
\AtBeginDocument{%
  \expandafter\renewcommand\expandafter\subsection\expandafter
    {\expandafter\@fb@secFB\subsection}%
  \newcommand\@fb@secFB{\FloatBarrier
    \gdef\@fb@afterHHook{\@fb@topbarrier \gdef\@fb@afterHHook{}}}%
  \g@addto@macro\@afterheading{\@fb@afterHHook}%
  \gdef\@fb@afterHHook{}%
}
\makeatother

\DeclareMathOperator{\wt}{wt}

\DeclareMathOperator{\Prob}{Prob}
\DeclareMathOperator{\size}{size}

\DeclareMathOperator{\lrs}{lrs}
\DeclareMathOperator{\rls}{rls}

\DeclareMathOperator{\weight}{weight}
\DeclareMathOperator{\shape}{shape}
\DeclareMathOperator{\lar}{large}
\DeclareMathOperator{\sma}{small}

\newcommand*\hole{\includegraphics[width=0.1in]{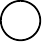}}
\newcommand*\heavy{\includegraphics[width=0.1in]{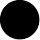}}
\newcommand*\light{\includegraphics[width=0.1in]{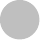}}

\begin{document}
\begin{center} {\Large{\sc Rhombic alternative tableaux and assembl\'{e}es of permutations}} \\
\vspace{0.1in}
Olya Mandelshtam\footnote{\noindent University of California Berkeley, Phone: +1 (949) 689-5748, E-mail: olya@math.berkeley.edu} and Xavier Viennot\footnote{\noindent LaBRI, Universit\'{e} Bordeaux, Phone: +33 (0)5 4000-6085, 
E-mail: viennot@xavierviennot.org}
 \end{center}

\begin{abstract}
In this paper, we introduce the \emph{rhombic alternative tableaux}, whose weight generating functions provide combinatorial formulae to compute the steady state probabilities of the two-species ASEP. In the ASEP, there are two species of particles, one \emph{heavy} and one \emph{light}, hopping right and left on a one-dimensional finite lattice with open boundaries. Parameters $\alpha$, $\beta$, and $q$ describe the hopping probabilities. The rhombic alternative tableaux are enumerated by the Lah numbers, which also enumerate certain \emph{assembl\'{e}es of permutations}. We describe a bijection between the rhombic alternative tableaux and these assembl\'{e}es. We also provide an insertion algorithm that gives a weight generating function for the assembl\'{e}es. Combined, these results give a bijective proof for the weight generating function for the rhombic alternative tableaux, which is also the partition function of the two-species ASEP at $q=1$.

\smallskip
\noindent \textbf{Keywords.} rhombic alternative tableaux, Lah numbers, assembl\'{e}es, ASEP, multispecies
\end{abstract}

 \section{Introduction}\label{sec_intro}
 
\emph{Rhombic alternative tableaux} were recently introduced by the authors in \cite{omxgv} to provide a combinatorial interpretation for the steady state probabilities of the two-species asymmetric simple exclusion process (ASEP). These tableaux are enumerated by the \emph{Lah numbers} ${n \choose r} \frac{(n+1)!}{(r+1)!}$. \emph{Assembl\'{e}es of permutations} are a generalization of permutations which are also enumerated by the Lah numbers. Our main result in this paper is a bijection that relates the rhombic alternative tableaux to the assembl\'{e}es while preserving certain statistics on the tableaux.

The ASEP is a model from statistical physics that describes the dynamics of interacting particles hopping right and left on a one-dimensional finite lattice with open boundaries. This model was originally introduced in the 1960's by biologists and mathematicians (see for example the survey papers \cite{blythe, chou} for the connection with biology). Since its introduction, the ASEP has received a lot of attention as an important example of a non-equilibrium process that exhibits boundary induced phase transitions, and for many other reasons, including its connection to orthogonal polynomials, the XXZ model, the formation of shocks, and more.

The classical ASEP with three parameters is defined by the following hopping probabilities: particles may enter at the left of the lattice with rate $\alpha$, they may exit at the right with rate $\beta$, and in the bulk the probability of hopping left is $q$ times that of hopping right. The stationary probability of a state of a Markov chain is the limit as time goes to infinity of the probability of observing that state. In other words, the vector of stationary probabilities of all the states is proportional to the left eigenvector with eigenvalue 1 of the transition matrix of the Markov chain. The ASEP is a Markov chain whose states are configurations of particles and holes on a finite lattice. Much past work has been devoted for finding combinatorial interpretations for the steady state probabilities of the ASEP in terms of various kinds of tableaux (permutation tableaux \cite{cw2007}, alternative tableaux  \cite{slides}, tree-like tableaux  \cite{treelike}, staircase tableaux  \cite{cw2011}). 

\begin{figure}[h]
\centering
\includegraphics[width=0.6\textwidth]{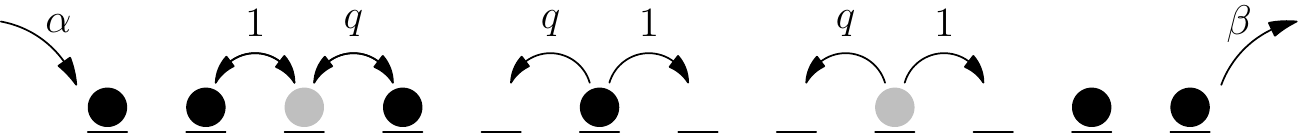}
\caption{The parameters of the two-species ASEP.}
\noindent
\label{2_parameters}
\end{figure}

The more general ASEP with five parameters $\alpha,\beta,\gamma,\delta$, and $q$ has a strong connection to the moments of the Askey-Wilson polynomials (see \cite{cssw,usw}), which are at the top of the hierarchy of the classical orthogonal polynomials in one variable. Recently, \cite{cantini,koornwinder} revealed a fascinating connection between moments of Koornwinder-Macdonald polynomials and a \emph{two-species generalization} of the ASEP with the same five parameters $\alpha,\beta,\gamma,\delta$, and $q$. Koornwinder polynomials are an important class of multi-variate orthogonal polynomials that generalize the Askey-Wilson polynomials. This two-species ASEP (studied in \cite{ayyer,duchi,uchiyama}, among others), has two species of particles, \emph{heavy} and \emph{light}, with the \emph{heavy} particles able to enter and exit at the boundaries. In the bulk, adjacent particles can swap, with rate 1 if the heavier particle is on the left, and rate $q$ if it is on the right, as in Figure \ref{2_parameters}. The recent connection with Koornwinder polynomials generated much interest in finding combinatorial objects that provide formulae for the two-species ASEP to expand upon the results for the single-species case. This was accomplished by the rhombic alternative tableaux \cite{omxgv} for the three-parameter ASEP, and the rhombic staircase tableaux \cite{cmw} for the most general five-parameter ASEP in recent work. 

In this paper we restrict ourselves to the case of the two-species ASEP with the parameters $\alpha,\beta$, and $q$. The two-species ASEP is a Markov chain on the states $X \in \{\heavy,\light,\hole\}^n$ with exactly $r$ \light's. (When $r=0$, we reduce to the original ASEP.) We call this set of states $B_n^r$, with $|B_n^r| = {n \choose r} 2^{n-r}$. See \cite{uchiyama, omxgv} for a more precise definitions of the transitions on $B_n^r$.


The rhombic alternative tableaux (RAT) are a two-species analog of the \emph{alternative tableaux} \cite{slides}, which interpret the probabilities of the original ASEP. These tableaux are certain fillings with $\alpha$'s, $\beta$'s, and $q$'s of rhombic tilings of closed shapes that correspond to states in $B_{n,r}$ of the two-species ASEP. We call the weight-generating function over the set of RAT corresponding to states in $B_{n,r}$ the \emph{partition function}, and we denote it by $\mathcal{Z}_{n,r}$. Enumeration of the RAT gives the following formula for $\mathcal{Z}_{n,r}$ \cite{omxgv}:
\begin{equation}\label{Z_nr}
\mathcal{Z}_{n,r} = (\alpha\beta)^{n-r} {n \choose r} \prod_{i=1}^{n-1} \left(\frac{1}{\alpha}+\frac{1}{\beta} + i\right).
\end{equation}

A canonical bijection introduced by the second author \cite{slides} called \emph{fusion-exchange} relates the alternative tableaux to permutations, as we can see in Figure \ref{fusion_examples} (a). At $r=0$, the RAT are precisely the alternative tableaux, and the set of assembl\'{e}es specializes to the set of permutations. In Section \ref{sec_bij} we extend the fusion-exchange bijection to relate the RAT to the assembl\'{e}es, as in Figure \ref{fusion_examples} (b). In Section \ref{sec_assemblee} we show that weighted enumeration of the assembl\'{e}es also yields the expression in Equation \eqref{Z_nr}. Consequently, we obtain a bijective proof for Equation \eqref{Z_nr}.

\begin{figure}[h]
\centering
\includegraphics[width=0.9\textwidth]{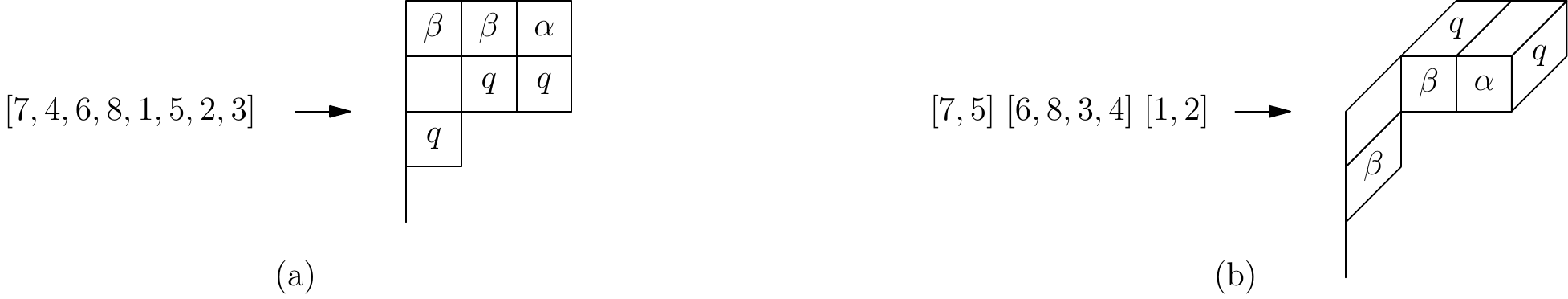}
\caption{(a) a permutation to an alternative tableau, and (b) an assembl\'{e}e to a RAT.}
\noindent
\label{fusion_examples}
\end{figure}

This paper is the full version of \cite{fpsac}, an extended abstract.

\textit{Acknowledgement.} The first author is grateful for the mentorship of Lauren Williams and Sylvie Corteel. The first author also acknowledges the hospitality of LIAFA, the Chateaubriand Fellowship, and the Berkeley-France fund, which made this work possible.

\section{Preliminaries}


Let $X \in \{\heavy,\light,\hole\}^n$ with exactly $r$ $\light$'s be a word describing a state of the two-species ASEP in $B_n^r$. 

\begin{defn}
The \emph{rhombic diagram} of \emph{type} $X$, where $X$ has $k$ \heavy's and $\ell$ \light's with $k+\ell+r=n$, is a closed shape whose northwest boundary is a path consisting of $\ell$ west edges followed by $r$ southwest edges followed by $k$ south edges. The southeast boundary is a path that is constructed by reading $X$ from left to right and drawing a west edge for a \hole, a southwest edge for a \light, and a south edge for a \heavy.  All the edges are of unit length, and the northeast and southwest corners of the northwest and southeast boundary are joined. This diagram is denoted by $\Gamma(X)$.
\end{defn}

\begin{defn}
A \emph{tiling} of $\Gamma(X)$, denoted by $\mathcal{T}$, is a covering of $\Gamma(X)$ by three types of rhombic tiles: the square, the tall rhombus, and the short rhombus, as pictured in Figure \ref{tiles}. 
\end{defn}

\begin{figure}[!ht]
 \begin{minipage}{0.3\linewidth}
  \centerline{\includegraphics[height=1.7in]{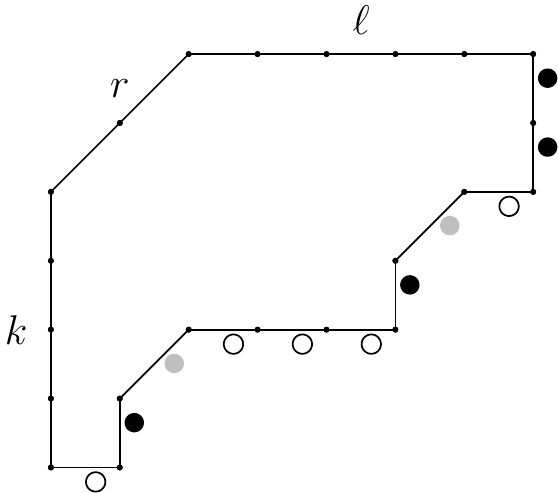}}
\centering
 \caption{$\Gamma(X)$ for the state $X=\protect\heavy\ \protect\heavy\ \protect\hole\ \protect\light\ \protect\heavy\ \protect\hole\ \protect\hole\ \protect\hole\ \protect\light\ \protect\heavy\ \protect\hole$.}\label{gamma_diag}
 \end{minipage}
\hfill
 \begin{minipage}{0.15\linewidth}
  \centerline{\includegraphics[height=1.5in]{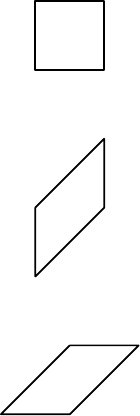}}
 \centering
  \caption{Tiles}\label{tiles}
 \end{minipage}
\hfill
 \begin{minipage}{0.45\linewidth}
  \centerline{\includegraphics[height=1.5in]{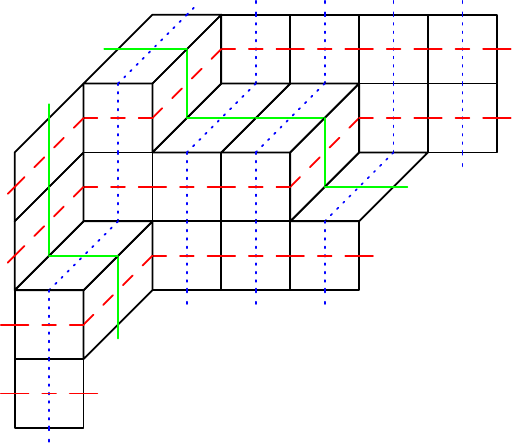}}
 \centering
  \caption{West-strips (red dashed), north-strips (blue dotted), and northwest-strips (green solid).}\label{strips}
 \end{minipage}
\end{figure}

\begin{defn}
A west-strip is a maximal connected set of tiles in which any two adjacent tiles share a vertical edge. Similarly, a north-strip is a maximal connected set of tiles in which any two adjacent tiles share a horizontal edge, and a northwest-strip is a maximal connected set of tiles in which any two adjacent tiles share a diagonal edge.
\end{defn}

These strips can be seen in Figure \ref{strips} by following the red dashed, blue dotted, and green solid lines.
In a tableau of type $X$ with $k$ \heavy's, $r$ \light's, and $\ell$ \hole's, the total number of west-strips is $k$, the total number of north-strips is $\ell$, and the total number of northwest-strips is $r$.

\begin{defn}
Let $\mathcal{T}$ be a tiling of $\Gamma(X)$. A \emph{rhombic alternative tableau} is a placement of $\alpha$'s, $\beta$'s, and $q$'s in the tiles of $\mathcal{T}$ such that the following three rules are satisfied:
\begin{enumerate}
\item[i.] Any tile left of and in the same west-strip as a $\beta$ is empty.
\item[i.] Any tile above and in the same north-strip as an $\alpha$ is empty.
\item[iii.] A tile that is not forced to be empty must contain an $\alpha$, $\beta$, or $q$.
\end{enumerate}
\end{defn}

We see an example of such a tableau in Figure \ref{RAT_example}. We denote by $R(\mathcal{T})$ the set of RAT associated to the tiling $\mathcal{T}$ of the diagram $\Gamma(X)$. When $r=0$, the definition above gives us precisely the alternative tableaux.

\begin{defn}\label{weight}
Let $T$ be a RAT of type $X$, where $X$ has $k$ \heavy's, $r$ \light's, and $\ell$ \hole's. Then the \emph{weight} of $T$ is defined as:
\[ \wt(T) = \alpha^k \beta^{\ell} \cdot (\mbox{ product of the symbols within } T) \]
\end{defn}

\begin{example}
For the tableau in Figure \ref{RAT_example}, we have $\wt(T) = \alpha^4\beta^4 \cdot \alpha^3 \beta^2 q^{15}$.
\end{example}

\begin{defn}
Let $T$ be a RAT of type $X$. We call the \emph{shape} of $T$ the shape of the rhombic diagram $\Gamma(X)$, and we denote it by $\shape(T)$.
\end{defn}

\begin{figure}[h!]
\centering
\includegraphics[height=1.5in]{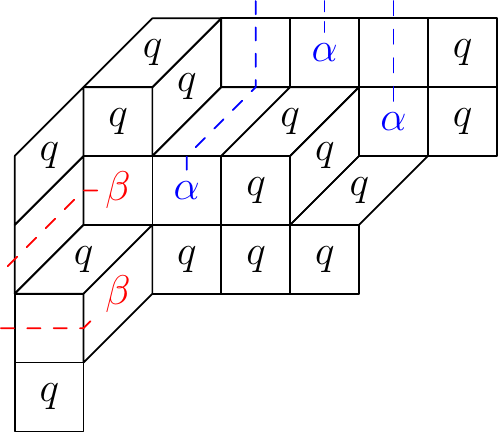}
\caption{A rhombic alternative tableau}\label{RAT_example}
\end{figure}

\begin{defn}
We define $\weight_{\mathcal{T}}(X)$ to be the weight generating function of all RAT with tiling $\mathcal{T}$, namely
\[ \weight_{\mathcal{T}}(X) = \sum_{T \in R({\mathcal{T})}} \wt(T). \]
\end{defn}

\begin{lem}[\cite{omxgv}]\label{flip_lem}
For two tilings $\mathcal{T}$ and $\mathcal{T}'$, 
\[\weight_{\mathcal{T}}(X) = \weight_{\mathcal{T}'}(X).\]
\end{lem}

Lemma \ref{flip_lem} enables us to define $\weight(X) = \weight_{\mathcal{T}}(X)$ for an arbitrary $\mathcal{T}$.


\begin{defn} We denote by $\mathcal{Z}_{n,r}$ the weight generating function over RAT of size $(n,r)$, namely:
$\mathcal{Z}_{n,r} = \sum_{X \in B_n^r}\weight(X)$.
\end{defn}

The following main result for the RAT is obtained in \cite{omxgv}.

\begin{thm}[\cite{omxgv}]\label{main_result}
Let $X \in B_n^r$ be a word representing a state of the two-species ASEP of size $(n,r)$. The steady state probability of state $X$ is
\[\Prob(X) = \frac{1}{\mathcal{Z}_{n,r}} \weight(X).\]
\end{thm}


\begin{defn}
An \emph{assembl\'{e}e} of size $(n,r)$ is a collection of $r$ nonempty ordered sets, or blocks, consisting of elements from $\{1,\ldots,n\}$, where the sets are all disjoint, and their union is $\{1,\ldots,n\}$. We call the last element in each block the \emph{block-end}, and we give a canonical ordering to the blocks such that the block-ends are decreasing from left to right. Define $\mathcal{A}_{n,r}$ to be the set of assembl\'{e}es of size $(n,r)$. 
\end{defn}

\begin{example}
$A= [2,10,12,7]\ [5,9,1,8,6]\ [3,11,4]$ is an example of an assembl\'{e}e of size $(12,3)$ in canonical order, and the block-ends are $[7,6,4]$.
\end{example}

From this point, an assembl\'{e}e is assumed to be in the canonical order.

\begin{defn}
Let $A$ be an assembl\'{e}e of size $(n,r)$. Define the block-end sequence $\mathbf{b} = [b_1,\ldots,b_r]$ to be the decreasing sequence of block-ends of $A$. Define the \emph{left-right sequence} to be the sequence of left to right maximal elements greater than $b_1$, and we call it $\lrs(A)$. Similarly, define the \emph{right-left sequence} to be the sequence of right to left maximal elements smaller than $b_r$, and we call it $\rls(A)$.
\end{defn}

\begin{example}
For $A = [2,10,12,7]\ [5,9,1,8,6]\ [3,11,4]$, we have $\lrs(A) = [12,11]$ and $\rls(A) = [3,2]$.
\end{example}


\begin{defn}
Suppose $x \in A$ is not a block-end. If $x+1$ is to the right of $x$ in $A$, then $x$ is an \emph{increase}. Otherwise, $x$ is a \emph{decrease}. By convention, $n+1$ is a decrease if it is not a block-end.
\end{defn}

\begin{example}
For $A=[2,10,12,7]\ [5,9,1,8,6]\ [3,11,4]$, the set of increases is $\{2,10,5,3\}$, and the set of decreases is $\{12,9,1,8,11\}$.
\end{example}

\begin{defn} 
Let $A$ be an assembl\'{e}e (in the canonical order) of size $(n+1,r+1)$. We define $X(A)$ to be a state in $B_{n}^{r}$ of the two-species ASEP. $X(A)$ is constructed by replacing each decrease with a \hole, each increase with a \heavy, and each block-end with a \light, with the last block-end omitted. We also define the \emph{shape} of $A$ to be $\shape(A)=\shape(X(A))$, which is also the shape of rhombic diagram $\Gamma(X(A))$.
\end{defn}

\begin{example}
For $A= [2,10,12,7]\ [5,9,1,8,6]\ [3,11,4]$, we have $X(A) = \heavy\ \heavy\ \hole\ \light\ \heavy\ \hole\ \hole\ \hole\ \light\ \heavy\ \hole$. 
\end{example}


\section{Bijection from rhombic alternative tableaux to assembl\'{e}es}\label{sec_bij}


The RAT are enumerated by the Lah numbers, which are indexed by $(n,r)$ and defined as ${n \choose r} \frac{(n+1)!}{(r+1)!}$. The Lah numbers also enumerate the assembl\'{e}es. In this section, we describe an RSK-style weight preserving bijection between assembl\'{e}es and the RAT, that generalizes the fusion-exchange algorithm of \cite{slides} for the alternative tableaux. 

\subsection{Assembl\'{e}es to RAT with the fusion-exchange algorithm}\label{sec_forward}

The goal of this section is to describe a bijection between $\mathcal{A}_{n+1,r+1}$ and RAT of size $(n,r)$. Our bijection is weight-preserving: an assembl\'{e}e $A$ of size $(n+1,r+1)$ with $|\lrs| = i$ and $|\rls|=j$ is mapped to some $T(A)$, which is a RAT such that $\wt(T(A))=\alpha^{n-r-i} \beta^{n-r-j}$ at $q=1$.

Let $A \in \mathcal{A}_{n+1,r+1}$. In Definition \ref{F_E} we describe the \emph{fusion-exchange algorithm} to build a rhombic alternative tableau $T(A)$ of type $X(A)$ and size $(n,r)$.  In this algorithm, we start with an arbitrary tiling on the rhombic diagram $\Gamma(X(A))$, and then send \emph{labels} through the edges of the tiles from the southeast- to the northwest boundary of $\Gamma(X(A))$ along the trajectories shown in Figure \ref{fig_init}. In each tile, from southeast to northwest, the two crossing labels will either \emph{exchange} or \emph{fuse}, which will determine the filling of that tile with precise rules as given in Definition \ref{F_E}.

%
%
%

\begin{figure}[h]
\centering
\includegraphics[height=2in]{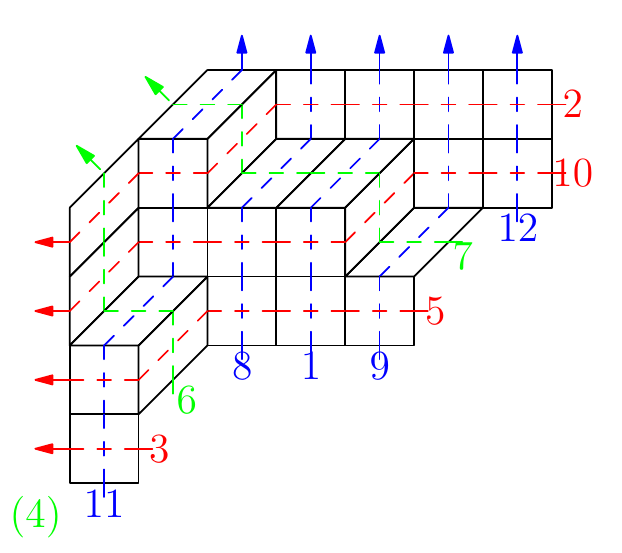}
\caption{The diagram $\Gamma(X(A))$ with an arbitrary tiling for $A = [2,10,12,7]\ [5,9,1,8,6]\ [3,11,4]$. The dashed lines show the trajectories followed by the labels on the southwest boundary edges as they travel towards the northwest boundary edges.}\label{fig_init}
\noindent
\end{figure}

\begin{defn}
A \emph{label} is a (possibly empty) set of consecutive integers. For two disjoint nonempty labels $A$ and $B$, we say $B \succ A$ if for the smallest $i \in B$ and the largest $j \in A$, $i=j+1$.
\end{defn}

Note that if $B \succ A$, then $A \cup B$ is also a label. We denote the empty label by $\emptyset$. By convention, $A \nsucc \emptyset$ for any label $A$ (including when $A=\emptyset$).

To simplify notation and minimize the number of cases to describe, we define the \emph{east edge} of a tile to be its east-most edge, i.e. the right-most vertical edge of a square or tall tile, or the right-most diagonal edge of a short tile. Similarly, the \emph{south edge} is its south-most edge, i.e. the bottom-most horizontal edge of a square or short tile, or the bottom-most diagonal edge of a tall tile. The \emph{west-} and \emph{north edges} are defined accordingly.

\begin{defn}\label{F_E} \textbf{(Fusion-Exchange Algorithm.)}

\textbf{Initiation.} We set the label of every southeast edge of $\Gamma(X(A))$ to be the corresponding element of $A$, omitting the last block-end. For example, see Figure \ref{fig_init}.

\begin{figure}[!ht]
  \centerline{\includegraphics[height=3in]{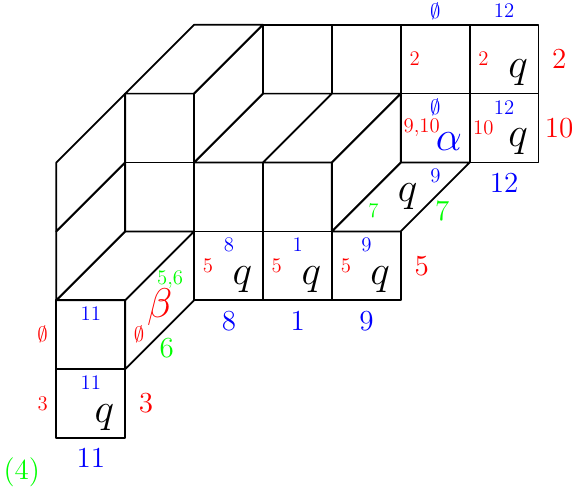}}
 \centering
  \caption{The first 11 steps of the fusion-exchange algorithm, as initiated in Figure \ref{fig_init}}\label{fig_steps}
\end{figure}

\textbf{Step.} Let a tile have label $A$ on its east edge and label $B$ on its south edge, with west and north edges unlabeled. Then a step of the algorithm involves labeling the west and north edge, and possibly placing $\alpha$, $\beta$, or $q$ in the tile according to the following cases.
\begin{enumerate}
\item[R(I)] $A \succ B$ and $B$ is labeling a horizontal edge. Then the west edge receives the label $B \cup A$, the north edge receives the label $\emptyset$, and an $\alpha$ is placed in the tile.
\item[R(II)] $B \succ A$ and $A$ is labeling a vertical edge. Then the north edge receives the label $A \cup B$, the west edge receives the label $\emptyset$, and a $\beta$ is placed in the tile.
\item[R(III)] All other cases, i.e. either $A \nsucc B$ and $B \nsucc A$, or $A \succ B$ and $B$ is labeling a diagonal edge, or $B \succ A$ and $A$ is labeling a diagonal edge. Then the labels simply pass through each other, and the west edge receives the label $A$ and the north edge receives the label $B$. A $q$ is placed in the tile if neither $A$ or $B$ equal $\emptyset$.
\end{enumerate}

Figure \ref{rules} illustrates the above rules applied to each tile and Figure \ref{fig_steps} shows some first steps of the algorithm.

\textbf{Termination.} The algorithm is complete once every edge of $\Gamma(X(A))$ has received a label. Figure \ref{fig_fusion_full} shows a complete example. After erasing the labels on the edges, we obtain a rhombic alternative tableau $T(A)$ of type $X(A)$ and size $(n,r)$. The following lemma shows $T(A)$ is a valid RAT.
\end{defn}

\begin{figure}[h]
\centering
\includegraphics[width=\linewidth]{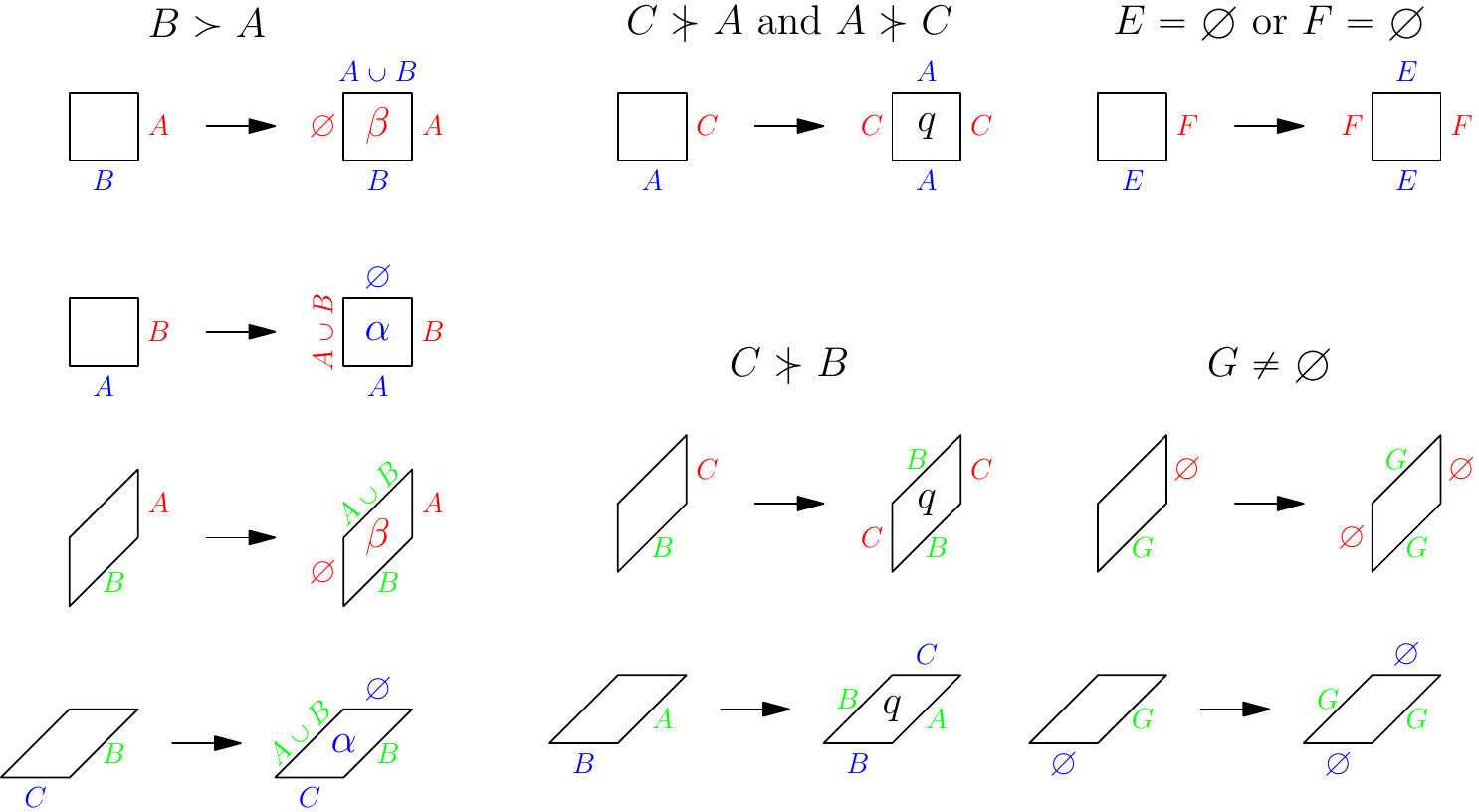}
\caption{The fusion-exchange rules for labels $A, B, C, E, F, G$ applied to all possible tiles.}
\noindent
\label{rules}
\end{figure}

\begin{lem}\label{forward_lemma} 
Let $A\in \mathcal{A}_{n+1,r+1}$ be an assembl\'{e}e, and let $\Gamma(X(A))$ be a rhombic diagram with arbitrary tiling $\mathcal{T}$. Then the fusion-exchange algorithm results in a tableau $T(A)$, which is a filling with $\alpha$'s and $\beta$'s of $\Gamma(X(A))$ such that:
\begin{enumerate}
\item[(i.)] A tile in the same north-strip and above an $\alpha$ is empty.
\item[(ii.)] A tile in the same west-strip and left of a $\beta$ is empty.
\item[(iii.)] A tile that is not forced to be empty must contain an $\alpha$, $\beta$, or a $q$.
\end{enumerate}
\end{lem}

Lemma \ref{forward_lemma} implies that $T(A)$ is a RAT of size $(n,r)$. Furthermore, we have the following theorem about the weight of $T(A)$.

\begin{figure}[h]
\centering
\includegraphics[width=0.8\linewidth]{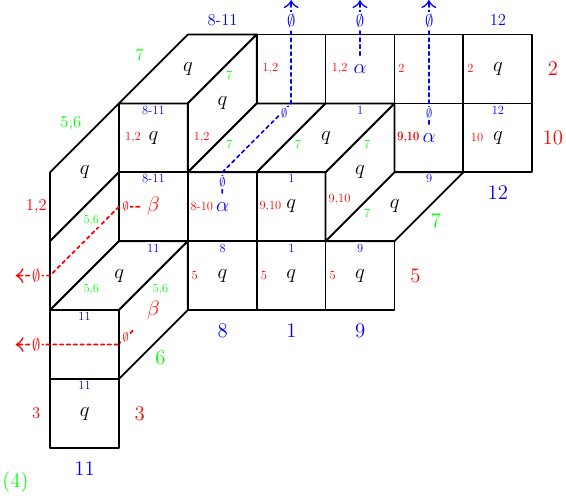}
\caption{The complete fusion-exchange algorithm to generate $T(A)$ for $A = [2,10,12,7]\ [5,9,1,8,6]\ [3,11,4]$. After erasing the labels on the edges, we obtain the tableau of Figure \ref{RAT_example}.}
\noindent
\label{fig_fusion_full}
\end{figure}

\begin{thm}\label{forward_thm}
Let $A\in \mathcal{A}_{n+1,r+1}$ be an assembl\'{e}e with $|\lrs(A)|=i$ and $|\rls(A)|=j$, and let $T(A)$ be the RAT obtained by applying the fusion-exchange algorithm to $A$. Then $\wt(T(A)) = \alpha^{n-r-i}\beta^{n-r-j}$ at $q=1$.
\end{thm}

\begin{proof}[Proof of Lemma \ref{forward_lemma}]
We observe the following, which is a direct consequence of R(I)-R(III) of Definition \ref{F_E}: 
\begin{itemize}
\item By R(III), if $\emptyset$ is labeling the south edge of a tile, after a step of the algorithm the labels pass through each other at that tile, and so the north edge of that tile also acquires the label $\emptyset$. Similarly, if $\emptyset$ is labeling the east edge of a tile, then after a step of the algorithm, the west edge of that tile also acquires the label $\emptyset$.
\item By R(I)-R(III), an $\alpha$, $\beta$, or $q$ can be placed in a tile if and only if both the south and east edges have labels that are not $\emptyset$.
\end{itemize}

By R(I), as soon as an $\alpha$ is placed in a tile $\tau_1$, its north edge (which is necessarily a horizontal edge) acquires the label $\emptyset$. Thus every horizontal edge above $\tau_1$ in the north-strip containing $\tau_1$ is labeled with $\emptyset$. Consequently, the algorithm must leave empty every tile in the same north-strip and above a tile containing an $\alpha$.
 
Similarly, by R(II), as soon as a $\beta$ is placed in a tile $\tau_2$, its west edge (which is necessarily a vertical edge) acquires the label $\emptyset$. Thus every vertical edge left of $\tau_2$ in the west-strip containing $\tau_2$ is labeled with $\emptyset$. Consequently, the algorithm must leave empty every tile in the same west-strip and left of a tile containing a $\beta$.

Finally, if neither the south or east edge of a tile is labeled by $\emptyset$, the tile must get an $\alpha$, $\beta$, or $q$.
\end{proof}

\begin{defn}
We fix our notation from this point onwards.
\begin{itemize}
\item A \emph{label at initiation} is an element of $A$ that is labeling an edge on the southeast boundary, at the initiation of the algorithm. 
\item If $x$ is the label of an edge at initiation, we can also call its corresponding strip \emph{the $x$-strip}. 
\item A \emph{label at termination} is the label $J$ that is labeling an edge on the northwest boundary corresponding to that strip at the termination of the algorithm. 
\item We say a label $A$ can \emph{absorb} a label $B$ if $A \succ B$. 
\item For a label $L$, we set $L_{\max}=\max\{x \in L\}$.
\end{itemize}
\end{defn}

\begin{example}
In Figure \ref{fig_fusion_full}, 2 is the label at initiation of the top-most west-strip (in other words, the 2-strip), and 7 is the label at initiation of the top-most northwest strip (in other words, the 7-strip). The 11-strip has label $(8,9,10,11)$ at termination, and the 5-strip has label $\emptyset$ at termination.
\end{example}

Note that for an assembl\'{e}e $A$ if $x$ is to the right of $y$ at initiation of $T(A)$, it is the same as saying $x$ is to the \emph{left} of $y$ in $A$.


To prove Theorem \ref{forward_thm}, we use the lemmas that follow. In particular, the following lemma is central to all of our proofs.

\begin{lem}\label{max_lemma}
If some edge $e$ of tableau $T$ has label $L \neq \emptyset$, and $x=L_{\max}$, then $e$ is in the $x$-strip.
\end{lem}

\begin{proof}
We provide a simple inductive proof on the steps of the fusion-exchange algorithm. This is trivially true at initiation. Now we observe that the rules of Definition \ref{F_E} imply the following:
\begin{enumerate}
\item If two labels exchange, they continue labeling edges in their respective strips in the next step of the algorithm, so the claim continues to be true.
\item Suppose that at a step of the algorithm, two labels $J$ and $L$ fuse in some tile $t$, with $J \succ L$. Then $J\cup L$ becomes the label of the opposite edge in the $J_{\max}$-strip, while the opposite edge of the $L_{\max}$-strip acquires the label $\emptyset$. Since $J_{\max}=(J \cup L)_{\max}$, the claim continues to be true after this step.
\end{enumerate}
Since this holds for every step of the algorithm, the lemma follows.
\end{proof}

The following lemmas establish an order on the labels at termination.

\begin{lem}\label{label_order}
Let $\{b_1,\ldots,b_{r+1}\}$ be the set of block ends in decreasing order of an assembl\'{e}e of size $(n+1,r+1)$. Then the union of labels of the horizontal edges on the northwest boundary is $\{b_1+1,\ldots,n+1\}$, the union of labels of the vertical edges is $\{1,\ldots,b_{r+1}-1\}$, and the union of the labels of the diagonal edges is $\{b_{r+1}+1,\ldots,b_1\}$.
\end{lem}

\begin{proof}
(1.) First, diagonal edge labels cannot contain any elements greater than $b_1$, since the $b_i$'s are never absorbed by a larger label. They also cannot contain any elements smaller than $b_{r+1}$ since no $b_i$ for $1 \leq i \leq r$ can absorb such elements, and $b_{r+1}$ is the auxiliary label and is not present in the tableau. 

(2.) Now suppose $L$ is a label with $L>b_1$ that is labeling a vertical edge at termination. Thus $L_{\max}$ is an \emph{increase}, and $L_{\max}+1$ is to its left at initiation. Now $L_{\max}+1$ is the smallest element in some label $J$ with $J \succ L$. Since $L_{\max}$ is never absorbed, it cannot cross the trajectory of $L_{\max}+1$, implying that $J$ is labeling a vertical edge south of $L$ at termination. We repeat the argument for $J$ until we must recursively conclude that some label $K$ with $K_{\max}=n+1$ is labeling a vertical edge at termination - a contradiction since $K$ is in the $n+1$-strip, and $n+1$ is a decrease. Thus any $L>b_1$ must necessarily be labeling a horizontal edge at termination. 

(3.) We have an analogous argument for the labels smaller than $b_{r+1}$. Suppose $L$ is a label with $L<b_{r+1}$ that is labeling a horizontal edge at termination. Thus $L_{\max}$ is a \emph{decrease}, and $L_{\max}+1$ is to its right at initiation. Now $L_{\max}+1$ is the smallest element in some label $J$ with $J \succ L$. Since $L_{\max}$ is never absorbed, it cannot cross the trajectory of $L_{\max}+1$, implying that $J$ is labeling a horizontal edge east of $L$ at termination. We repeat the argument for $J$ until we must recursively conclude that some label $K$ with $K_{\max}=b_{r+1}-1$ is labeling a horizontal edge at termination - a contradiction since $K$ is in the $b_{r+1}$-strip, and $b_{r+1}$ is an increase. Thus any $L>b_1$ must necessarily be labeling a vertical edge at termination. 

(4.) Now suppose $L$ is a label with $b_{i+1}<L<b_{i}$. If we suppose $L$ is labeling a vertical edge at termination, an almost identical argument to (2.) implies that $L$ must be absorbed by some label $K$ with $K_{\max}=b_i$. And if we suppose $L$ is labeling a horizontal edge at termination, we obtain the same with an argument almost identical to (3.). These two cases contain contradictions, and thus no such $L$ can be a label at termination. 

Since every element in $\{1,\ldots,n+1\}$ save for $b_{r+1}$ must appear in some label at termination, the lemma follows.
\end{proof}

\begin{lem}\label{label_order_lem}
The labels at termination are ordered. Let $L$ and $J$ be two labels with $J \succ L$. Then the following occurs:
\begin{enumerate}
\item[(i.)] If $L$ and $J$ are labeling horizontal edges, then $J$ is east of $L$.
\item[(ii.)] If $L$ and $J$ are labeling vertical edges, then $J$ is south of $L$.
\item[(iii.)] If $L$ and $J$ are labeling diagonal edges, then $J$ is north of $L$.
\end{enumerate}
\end{lem}

\begin{proof}
(i.) Suppose the contrary, that $J$ is west of $L$. Let $x=L_{\max}$. $x$ is horizontal, so $x+1$ must be to its \emph{right} at initiation. But $x+1 \in J$, so if $J$ is west of $L$, the trajectory of $x+1$ necessarily intersects the $x$-strip, which would result in $x$ being absorbed. By Lemma \ref{max_lemma}, this is a contradiction, so $J$ must be west of $L$.

(ii.) Suppose the contrary, that $J$ is north of $L$. Let $x=L_{\max}$. $x$ is vertical, so $x+1$ must be to its \emph{left} at initiation. But $x+1 \in J$, so if $J$ is north of $L$, the trajectory of $x+1$ necessarily intersects the $x$-strip, which would result in $x$ being absorbed. By Lemma \ref{max_lemma}, this is a contradiction, so $J$ must be south of $L$.

(iii.) This follows from Lemma \ref{max_lemma} and the fact that the block-ends are in decreasing order from northeast to southwest.
\end{proof}

\begin{lem}\label{blockend_lemma}
If $x = b_i$ is a block-end with $i<r+1$, then the $x$-strip has label $C \neq \emptyset$ at termination.
\end{lem}

\begin{proof}
According to R(III) of the fusion-exchange algorithm, a label on a diagonal edge can never be absorbed by another label. The Lemma follows by the proof of Lemma \ref{max_lemma}.
\end{proof}

\begin{remark}
Note that Lemmas \ref{label_order_lem} and \ref{blockend_lemma} imply the complete ordering at termination as shown in Figure \ref{fig_order}. Let the nonempty labels of the north-strips at termination be $K_1,\ldots,K_v$ from left to right. Let the nonempty labels of the west-strips at termination be $I_1,\ldots,I_u$ from bottom to top. Let the labels of the north-west strips at termination be $J_1,\ldots,J_t$ from top to bottom. Then $K_v \succ \cdots \succ K_1 \succ J_1 \succ \cdots \succ J_t \succ (b_{r+1}) \succ I_1 \succ \cdots \succ I_u$. See Figure \ref{fig_fusion_full} for an example.
\end{remark}

\begin{figure}[h!]
  \centerline{\includegraphics[width=2in]{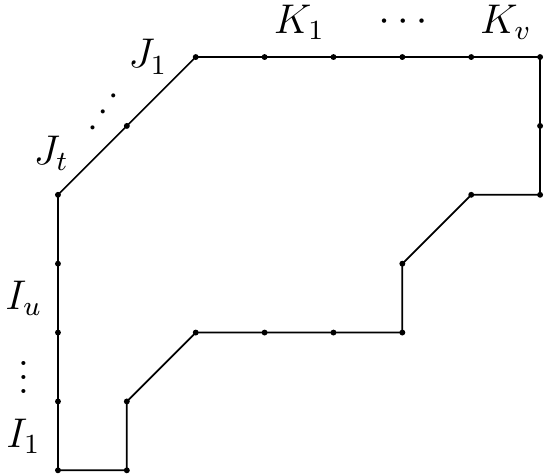}}
 \centering
  \caption{The labels appearing at termination, with $K_v \succ \cdots \succ K_1 \succ J_1 \succ \cdots \succ J_t \succ (b_{r+1}) \succ I_1 \succ \cdots \succ I_u$.}\label{fig_order}
\end{figure}


\begin{lem}\label{lrs_rls}
Let $A$ be an assembl\'{e}e of size $(n+1,r+1)$.
\begin{enumerate}
\item[(i.)] The largest elements of the labels of the horizontal edges of $T(A)$ at termination are precisely the elements of $\lrs(A)$ in decreasing order from left to right in $A$.
\item[(ii.)] The largest elements of the labels of the vertical edges of $T(A)$ at termination are precisely the elements of $\rls(A)$ in decreasing order from right to left in $A$.
\end{enumerate}
\end{lem}

\begin{proof}
(i.) Let $L$ be a label of a horizontal edge at termination. $L$ is labeling the $L_{\max}$-strip. If $L_{\max}=n+1$, it is in $\lrs(A)$ by default. For $L_{\max}<n+1$, let $J$ be a label with $J> L$. By Lemma \ref{label_order_lem} we have that at termination $J$ is labeling a horizontal edge east of the $L_{\max}$-strip. Labels travel from southeast to northwest, so the elements of $J\backslash J_{\max}$ are all to the right of $J_{\max}$ at initiation. Since this is true for all $J>L$, any $y>x$ must be to the \emph{left} of $L_{\max}$ in $A$. Thus $L_{\max} \in \lrs(A)$. 

For the converse, for any $K$ labeling a horizontal edge at termination and for any $x < K_{\max}$, by the same argument $x$ must have $K_{\max}$ to its \emph{right} in $A$. Consequently $x \in \lrs(A)$ if and only if $x = K_{\max}$ for some label $K$ of a horizontal edge at termination. The elements of $\lrs{A}$ are in decreasing order from left to right in $A$ also by Lemma \ref{label_order_lem}.

(ii.) We have the same argument as for (i.) after replacing $n+1$ by $b_{r+1}-1$, $\lrs(A)$ by $\rls(A)$, horizontal by vertical, east by south, and swapping left with right. Let $L$ be a label of a vertical edge at termination. $L$ is labeling the $L_{\max}$-strip. If $L_{\max}=b_{r+1}-1$, it is in $\rls(A)$ by default. For $L_{\max}<b_{r+1}-1$, let $J$ be a label with $J> L$. By Lemma \ref{label_order_lem} we have that at termination $J$ is labeling a vertical edge south of the $L_{\max}$-strip. Labels travel from southeast to northwest, so the elements of $J\backslash J_{\max}$ are all to the left of $J_{\max}$ at initiation. Since this is true for all $J>L$, any $y>x$ must be to the \emph{right} of $L_{\max}$ in $A$. Thus $L_{\max} \in \rls(A)$. 

For the converse, for any $K$ labeling a vertical edge at termination and for any $x < K_{\max}$, by the same argument $x$ must have $K_{\max}$ to its \emph{left} in $A$. Consequently $x \in \rls(A)$ if and only if $x = K_{\max}$ for some label $K$ of a vertical edge at termination. The elements of $\rls{A}$ are in decreasing order from right to left in $A$ also by Lemma \ref{label_order_lem}.
\end{proof}

Finally we obtain the proof of our main result for this section, Theorem \ref{forward_thm}.

\begin{proof}[Proof of Theorem \ref{forward_thm}]
By R(I) a north-strip of $T(A)$ contains an $\alpha$ if and only if it is labeled by $\emptyset$ at termination, and by $R(II)$ a west-strip of $T(A)$ contains a $\beta$ if and only if it is labeled by $\emptyset$ at termination. By Lemma \ref{lrs_rls}, every north-strip with a nonempty label corresponds to an element of $\lrs(A)$, and every west-strip with a nonempty label corresponds to an element of $\rls(A)$.

As given in Definition \ref{weight}, the weight of a rhombic alternative tableau $T$ of size $(n,r)$ is equivalent to the following:
\[\wt(T) = (\alpha\beta)^{n-r}\alpha^{-i}\beta^{-j},\]
where $i$ is the number of $\alpha$-free north-strips (i.e. north-strips not containing $\alpha$) and $j$ is the number of $\beta$-free west-strips (i.e. west-strips not containing $\beta$). Equivalently, we have $i=|\lrs(A)|$ and $j=|\rls(A)|$, which completes the proof of the theorem.
\end{proof}

\subsection{Rhombic alternative tableaux to assembl\'{e}es}\label{sec_reverse}

In this section, we describe an algorithm that serves as the inverse of the fusion-exchange algorithm of Section \ref{sec_forward}, thereby establishing a weight-preserving bijection from RAT of size $(n,r)$ to $\mathcal{A}_{n+1,r+1}$. Let $T$ be a RAT of type $X$ with arbitrary tiling $\mathcal{T}$ and size $(n,r)$, such that $X$ has $k$ \heavy's and $\ell$ \hole's (with $k+\ell+r=n$). We will construct from $T$ an assembl\'{e}e $A(T)$ of size $(n+1,r+1)$.


First, a forest of crossing binary trees is associated to $T$, starting with a network of lines passing through the west-, north- and northwest- strips as in Figure \ref{strips}. For each west-strip, a line is drawn through the midpoints of the vertical edges of each tile of the strip, indicated by red. Corresponding lines for the north-strips and the north-west strips are drawn, indicated by blue and green respectively.

We note here that the colors red, blue, and green are not essential to the algorithm, but they greatly facilitate its visualization. In particular, red always corresponds to west-strips, blue to north-strips, and green to northwest-strips.

Now, for each tile containing an $\alpha$, the section of the blue line north of that $\alpha$ is removed. Similarly, for each tile containing a $\beta$, the section of the red line west of that $\beta$ is removed. As seen in Figure \ref{reverse_init}, we obtain a forest of binary trees, where the vertices (or branching nodes) of the trees are in the tiles containing $\alpha$ or $\beta$. Moreover, a tile contains two intersecting branches (either from different trees or from the same tree) if and only if it contains a $q$. Finally, each tree has a \emph{root} starting on the northwest border of $T$. We say the roots belong to the classes \emph{north, west}, or \emph{northwest}, corresponding to the types of strips at which the roots are located. 

\begin{defn} We identify each edge on the southeast boundary of $T$ with an \emph{external vertex}, which corresponds to a leaf of some binary tree. In a tableau of size $n$, there are $n$ external vertices. 
\end{defn}

By our convention, we also add a special trivial northwest tree/root to the southwest point of our tableau. Note that by this construction, the forest associated to a RAT of size $(n,r)$ with $i$ $\alpha$-free north-strips and $j$ $\beta$-free west-strips, has $i$ trees with a north root, $j$ trees with a west root, and $r+1$ trees with a northwest root (including the special trivial tree). 

\begin{example}
In Figure \ref{reverse_init}, the trivial root is denoted by the label ``$(4)$''. In this same example, the north roots have labels $\{(8,9,10,11), (12)\}$, the west roots have labels $\{(1,2),(3)\}$, and the northwest roots have labels $\{(5,6), (7), (4)\}$. 
\end{example}

\begin{remark}
The binary trees are drawn on the tiling $\mathcal{T}$, and a ``flip'' on the tiling corresponds to a certain local move on the branches of the tree, similar to a Yang-Baxter move. 
\end{remark}

%

\begin{figure}[!h]
  \centerline{\includegraphics[width=3in]{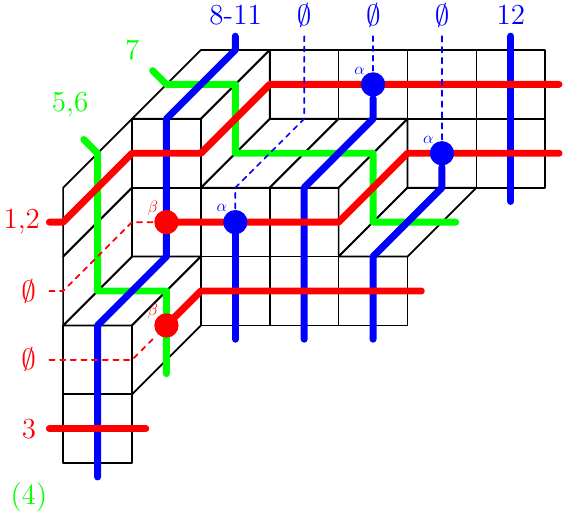}}
\centering
 \caption{The forest of binary trees corresponding to the rhombic alternative tableau of Figure \ref{RAT_example}.}\label{reverse_init}
 \end{figure}

The roots of the three classes of binary trees, including the special trivial green tree, are totally ordered according to the following definition.

\begin{defn}\label{ordering} The roots are in increasing order starting with the $j$ red roots from top to bottom, then the $r+1$ green roots from southwest to northeast, and finally the $i$ blue roots from left to right. Each root of a binary tree $\mathcal{B}$ is labeled by some label $B$ according to the three following conditions:

\begin{itemize}
\item The size of $B$ is one if it is the special trivial tree, otherwise it equals the number of external vertices of $\mathcal{B}$.
\item  For any two consecutive root edges $R$ and $R'$ (where $R'$ is greater than $R$ in the total order) the corresponding labels $B$ and $B'$ satisfy  $B' \succ B$.
\item  The union of the labels of the root edges is $\{1,\ldots,n+1\}$.
\end{itemize}
\end{defn}


Figure \ref{reverse_init} shows an example of the labeling of the root edges. Note that these labels precisely correspond to the labels of $T(A)$ of Figure \ref{fig_fusion_full} at termination of the exchange fusion algorithm.

%

\begin{defn}[Label-passing algorithm]
After assigning the labels to the roots of the forest associated to a RAT, the labels are passed from the roots along the branches to the southeast. Whenever a label $B$ is at a vertex $v$, it is split into two disjoint labels $C$ and $D$ such that $C \cup D = B$, and $C$ and $D$ are passed along the branches going southeast from $v$. The branch that is passed the label $C$ must be connected to $|C|$ external vertices, and the branch that is passed the label $D$ must consequently be connected to $|D|$ external vertices. Setting $D \succ C$, we have the following cases for $v$.
\begin{enumerate}
\item[I.] If $v$ is in a northwest-strip, then the branch going out of $v$ in the southeast direction is passed the label $D$. (That is, the larger label is always passed down the green northwest line.)
\item[II.] Otherwise, if $v$ corresponds to an $\alpha$, then the branch going out of $v$ in the east direction is passed the label $D$, and the branch going out of $v$ in the south direction is passed the label $C$. (That is, if $v$ corresponds to an $\alpha$, the smaller label is always passed down the blue west line.)
\item[III.] Similarly, if $v$ corresponds to a $\beta$, then the branch going out of $v$ in the south direction is passed the label $D$, and the branch going out of $v$ in the east direction is passed the label $C$. (That is, if $v$ corresponds to a $\beta$, the smaller label is always passed down the red north line.)
\end{enumerate}

The algorithm is complete once every external vertex has received a label that consists of a single integer. An assembl\'{e}e is then obtained by reading these external vertex labels from northeast to southwest, and assigning the labels of the green diagonal external vertices to be the block-ends. The final block-end is assigned to be the integer in $\{1,\ldots,n+1\}$ that is the label of the special trivial root added at the beginning.
\end{defn}

 \begin{figure}[h!]
  \centerline{\includegraphics[width=0.8\linewidth]{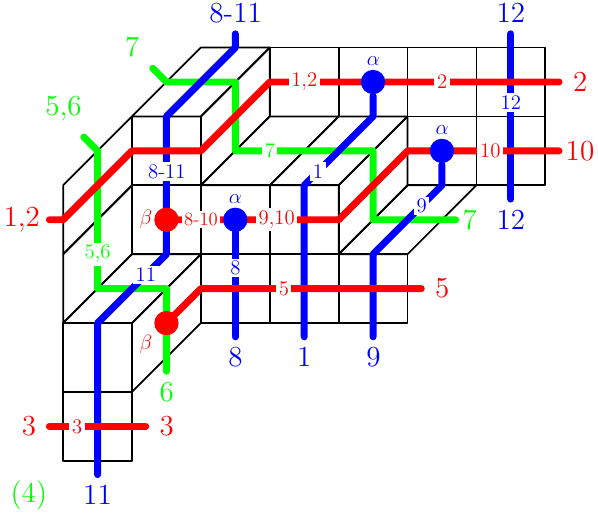}}
\centering
 \caption{The complete label-passing algorithm applied to the tableau $T$ of Figure \ref{RAT_example}, with the algorithm initiated in Figure \ref{reverse_init}, and resulting in $A(T)=[2,10,12,7]\ [5,9,1,8,6]\ [3,11,4]$.}\label{exfusalgo}
 \end{figure}

We see an example of the completed algorithm in Figure \ref{exfusalgo}. Note that this example is essentially the same as the example for the fusion-exchange algorithm of Figure \ref{fig_fusion_full}. Comparing these two examples, one can see that the labels passed along the branches in label-passing are the same as the labels on the edges in fusion-exchange, and so we can identify the labels on the branches with the labels on the edges they pass through. With this identification, a step of the fusion-exchange algorithm has as its inverse a step of the label-passing algorithm, and vice versa. That implies that label-passing is well-defined as the inverse of the fusion-exchange algorithm. This is not hard to see, but we provide the proof below.


%
%
%
%

\begin{thm}\label{thm_reverse}
At its conclusion, the label-passing algorithm results in an assembl\'{e}e $A(T)$ that is read off the labels on the edges of the southeast boundary of $T$ from right to left, such that $\shape(A(T)) = \shape(T)$. Moreover, $\wt(T)=(\alpha\beta)^{n-r} \alpha^{-i}\beta^{-j}$ at $q=1$ where $i=|\lrs(A(T))|$ and $j=|\rls(A(T))|$.
\end{thm}

The following lemmas are needed for the proof of the theorem.

\begin{lem}\label{no_alpha}
Suppose a label $J$ is passed down a north-strip $\mathbf{n}$ from a tile $\mathbf{t}$ such that there are no $\alpha$'s below $\mathbf{t}$ in $\mathbf{n}$. Then at the completion of the label-passing algorithm, $J_{\max}$ is labeling the external vertex of $\mathbf{n}$.
\end{lem}

\begin{proof} A simple inductive proof suffices. If $J=\{x\}$, then the claim trivially holds. Now suppose our claim holds for all $|J|<k$. Let $|J|=k>1$. Then there must be at least one vertex south of $\mathbf{t}$ in $\mathbf{n}$ at which $J$ splits. This vertex must correspond to a $\beta$. Consequently, $J=A \cup B$ with $B \succ A$ for some $A$ and $B$, and $A$ is sent along the branch east of that vertex, and $B$ is sent along branch $\mathbf{n}$. Now $|B|<k$ and $x = \max\{y \in B\}$, so our claim holds by induction.
\end{proof}

Similarly, we have the following symmetric lemma, for which we omit the proof.

\begin{lem}\label{no_beta}
Suppose a label $J$ is passed along a west-strip $\mathbf{w}$ from a tile $\mathbf{t}$ such that there are no $\beta$'s to the right of $\mathbf{t}$ in $\mathbf{w}$. Then at the completion of the label-passing algorithm, $L_{\max}$ is labeling the external vertex of $\mathbf{w}$.
\end{lem}

\begin{proof}[Proof of Theorem \ref{thm_reverse}]
The four conditions below are sufficient to show $\shape(A(T))$, $|\lrs(A(T))|$, and $|\rls(A(T))|$ have the desired properties.
\begin{enumerate}
\item[i.] The label of every horizontal external vertex of the forest on $A(T)$ is a decrease of $A(T)$.
\item[ii.] The label of every vertical external vertex of the forest on $A(T)$ is an increase of $A(T)$.
\item[iii.] A north-strip of $T$ does not contain an $\alpha$ if and only if its external vertex label belongs to $\lrs(A(T))$.
\item[iv.] A west-strip of $T$ does not contain a $\beta$ if and only if its external vertex label belongs to $\rls(A(T))$.
\end{enumerate}

Suppose at the termination of the label-passing algorithm, $x$ is the label of a horizontal external vertex of the forest on $A(T)$. Let us call the north-strip containing $x$ the $x$-strip. We dissect two cases: if this strip contains $\alpha$, and if it does not contain $\alpha$.

If the $x$-strip contains an $\alpha$, we consider the tile $\mathbf{t}$ with that $\alpha$. At the vertex in $\mathbf{t}$, the label that enters the vertex from the root is $J = A \cup B$ with $B \succ A$ for some $A$ and $B$, and $B$ is passed to the right branch out of this vertex, and $A$ is passed to the south branch out of this vertex. Since there are no other $\alpha$'s south of $\mathbf{t}$, Lemma \ref{no_alpha} implies that $x = A_{\max}$. Now, the label-passing algorithm sends every element in $B$ to the northeast of the $x$-strip. Thus, since $x+1 \in B$, at termination $x+1$ is labeling an edge to the northeast of $x$, and so $x+1$ is to the \emph{left} of $x$ in $A$. $x$ is a decrease as desired, giving us (i).

If the $x$-strip does not contain an $\alpha$, Lemma \ref{no_alpha} implies that $x=J_{\max}$ where $J$ is the label of the root of the $x$-strip. Any elements larger than $x$ must be contained in the labels of roots of north-strips that are to the right of the $x$-strip. Since labels travel to the southeast, elements from these root labels necessarily end up labeling external vertices to the northeast of $x$. Thus any $y>x$ must be to the \emph{left} of $x$ in $A$. This implies $x$ is a decrease giving us (i.), and moreover that $x \in \lrs(A)$.

Note that $\wt(T)=(\alpha\beta)^{n-r}\alpha^{-i}\beta^{-j}$ implies $T$ has $i$ $\alpha$-free north-strips and $j$ $\beta$-free west-strips. From the above argument, we get for free that if $z \in J$ for label $J$ of some a root in the $x$-strip and $z< J_{\max}$, then $z \not\in \lrs(A(T))$. 
Thus we obtain in particular that the number of blue-rooted trees, which is also the number of $\alpha$-free north-strips in $T$, is $|\lrs(A(T))|$. Consequently $|\lrs(A(T))| = i$, giving us (iii).

The proofs for (ii) and (iv) are symmetric, so we omit them.
\end{proof}

\begin{thm}\label{thm_bijection}
Let $\tilde{A}$ be an assembl\'{e}e and let $\tilde{T}$ be a rhombic alternative tableau. Then
\begin{equation}\label{eq_A}
\tilde{A} = A(T(\tilde{A})),
\end{equation}
and
\begin{equation}\label{eq_T}
\tilde{T} = T(A(\tilde{T})).
\end{equation}
\end{thm}

This theorem implies that the exchange fusion algorithm is a bijection from RAT of size $(n,r)$ to $\mathcal{A}_{n+1,r+1}$, and its inverse is the label-passing algorithm.

\begin{proof}
By construction, the steps of the label-passing algorithm in which labels \emph{split} accomplish the reverse of the steps of the exchange-fusion algorithm in which labels \emph{fuse}. We see this by considering all four possible cases of splitting/fusion in Figure \ref{fig_bij_vertex}. A branch of the label-passing algorithm with label $A \cup B$ for $B \succ A$ enters vertex $v$ contained in tile $\mathbf{t}$. The branches going out of $v$ have respective labels $A$ and $B$. In each of the four cases, labeling the south and east edges of $\mathbf{t}$ with the labels passed along the branches through these edges is the unique possible labeling according to the fusion-exchange algorithm. 

At all other points of the two algorithms, the labels remain unchanged on their corresponding trajectories. 

\begin{figure}[!h]
\centerline{\includegraphics[width=0.8\linewidth]{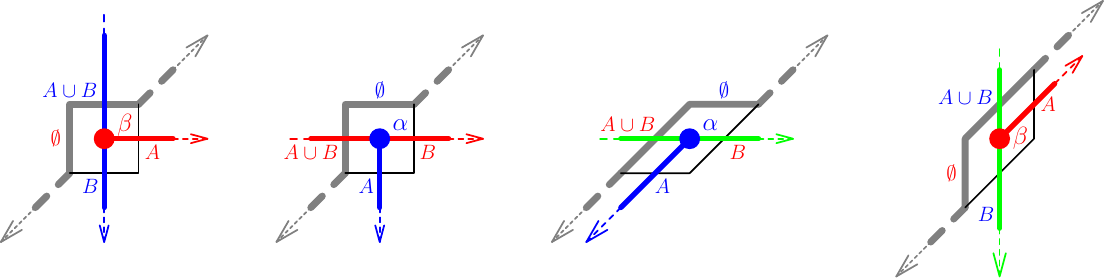}}
\centering
\caption{The splitting of $A \cup B$ in the label-passing algorithm at a vertex in tile $\mathbf{t}$ matches the fusion of two labels $B \succ A$ in the exchange fusion algorithm.}\label{fig_bij_vertex}
 \end{figure}
 
Furthermore by construction, at initiation of the label-passing algorithm, the labeling of the root edges matches precisely the ordering of labels at termination of the exchange fusion algorithm by Lemma \ref{label_order_lem}. These arguments, combined with Lemmas \ref{max_lemma}, \ref{no_alpha}, and \ref{no_beta} are enough to complete the proof.
\end{proof}

As a consequence of Theorem \ref{thm_bijection} and Theorem \ref{thm_reverse} according to the definitions of $\wt(T)$ and $\wt(A)$ for a rhombic alternative tableau $T$ and an assembl\'{e}e $A$, we obtain the following corollary.

\begin{cor}\label{cor_wt}
The fusion-exchange algorithm is a weight-preserving bijection with the label-passing algorithm as its inverse.
\end{cor}





\section{Weighted enumeration of assembl\'{e}es}\label{sec_assemblee}

In this section, we provide a weight-preserving bijection between assembl\'{e}es and a pair composed by the choice of a subset ${n \choose r}$ and an \emph{$r$--truncated subexceedant function} on $[n-r]$, which is enumerated by the product $(r+2)\ldots(n+1)$. This leads to a bijective proof of the following theorem.

\begin{thm}
Let $A \in \mathcal{A}_{n+1,r+1}$ be an assembl\'{e}e. Then
\[\sum_{A\ :\ \size(A)=(n+1,r+1)} \alpha^{-|\lrs(A)|}\beta^{-|\rls(A)|} = {n \choose r} \prod_{i=r}^{n-1} \left( \alpha^{-1}+\beta^{-1}+i \right).\]
\end{thm}

\begin{defn}
A \emph{subexceedant function} on $[n]=[1,\ldots,n]$ is a function $f\ :\ [n] \rightarrow [n]$ such that $f(i)\leq i$ for each $i \in [n]$. An \emph{$r$--truncated subexceedant function} is such a function $f$ with $f(i)\leq i+r+1$.
\end{defn}

We define an insertion algorithm as follows. Let $f$ be an $r$--truncated subexceedant function on $[n-r]$. We begin with $r+1$ horizontal green lines at heights $1$ through $r+1$ from bottom to top. For $i\in [n-r]$, we insert element $i$ to the right of element $i-1$ in position $f(i)$ relative to the $r+1$ green lines and the elements $[i-1]$ which have already been inserted. In other words, if $f(i)=1$, then the $i$'th element is inserted below the other elements and the green lines, and if $f(i)=k$, then the $i$'th element is inserted above the element at height $k-1$.

Once all elements have been inserted, each is assigned a value from $[n+1]$ that corresponds to its height relative to the other elements and the green lines. Finally, a point is chosen on each of the green lines such that the points are located from top to bottom when read from left to right, and the point on the bottom-most green line is fixed to be to the right of the last inserted element. This selection fixes the locations of the block-ends in the assembl\'{e}e. Figure \ref{insertion} shows an example of this insertion of size $(11,3)$, where the $3$--truncated subexceedant function $f$ is defined by $f(1)=3,f(2)=5,f(3)=2,f(4)=6,f(5)=1,f(6)=9,f(7)=2,f(8)=1$.

It is easy to check that an $r$--truncated subexceedant function $f$ combined with a choice of the positions of the points on the green lines results in an assembl\'{e}e $A \in \mathcal{A}_{n+1,r+1}$. Note that there are ${n \choose r}$ ways to choose the points on the green lines, and each of those choices results in a distinct assembl\'{e}e. 

For the $r$-truncated subexceedent function $f$, the inserted elements are given colors and weights as follows:
\begin{itemize}
\item When $f(i)=1$, the $i$'th inserted element is colored \emph{red} and given weight $\beta^{-1}$.
\item When $f(i)=i+r+1$, the $i$'th inserted element is colored \emph{blue} and given weight $\alpha^{-1}$.
\item Otherwise, the element is colored black and given weight 1.
\end{itemize}

Notice that if $f(i)=1$, this means the $i$'th inserted element is below the green lines and every element to its left, and if $f(i)=i+r+1$, this means the $i$'th inserted element is above the green lines and every element to its left.

\begin{figure}[h]
 \begin{minipage}{0.46\linewidth}
  \centerline{\includegraphics[height=3.1in]{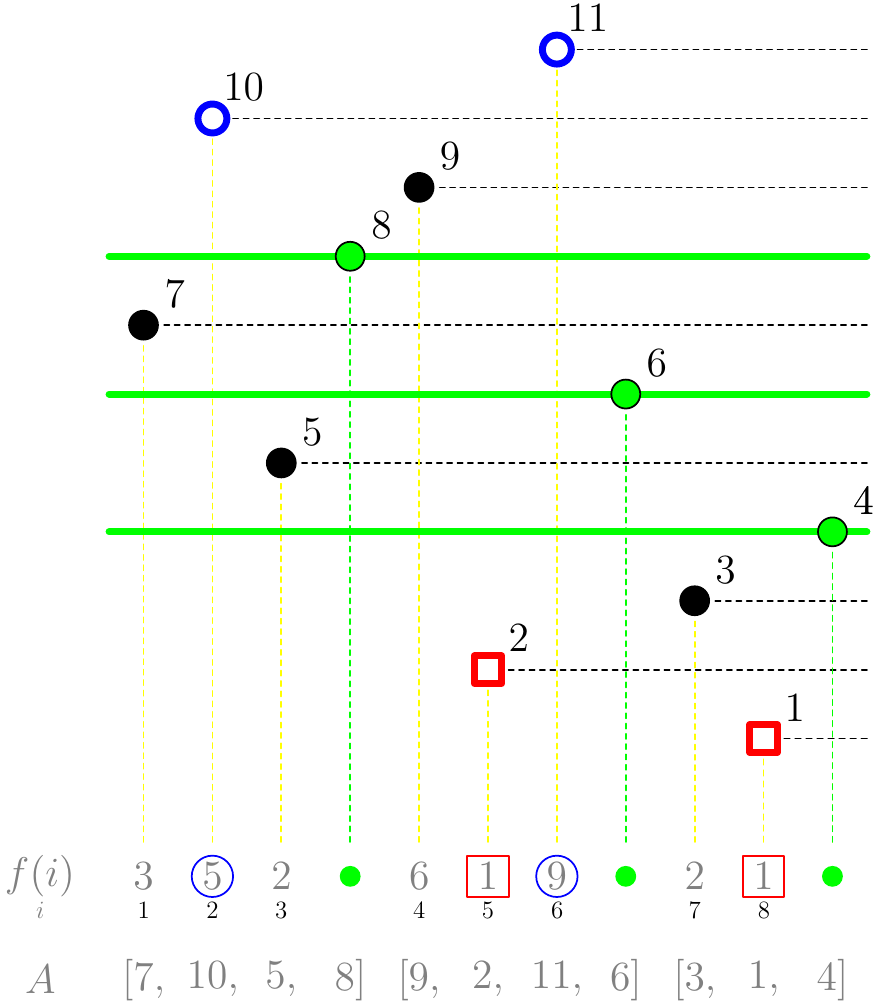}}
  \centering
\caption{Insertion given by $f(i)$ for $i=1,\ldots,8$ to obtain $A$, with blue circles and red squares representing elements with weights $\alpha^{-1}$ and $\beta^{-1}$ respectively, and the green lines and circles on them describing the block-ends.} \label{insertion}
 \end{minipage}
\begin{minipage}{0.08\linewidth}
  \centerline{\includegraphics[width=\linewidth]{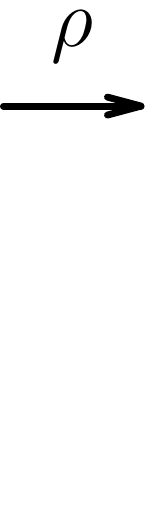}}
 \centering
 \end{minipage}
 \begin{minipage}{0.46\linewidth}
  \centerline{\includegraphics[height=3.1in]{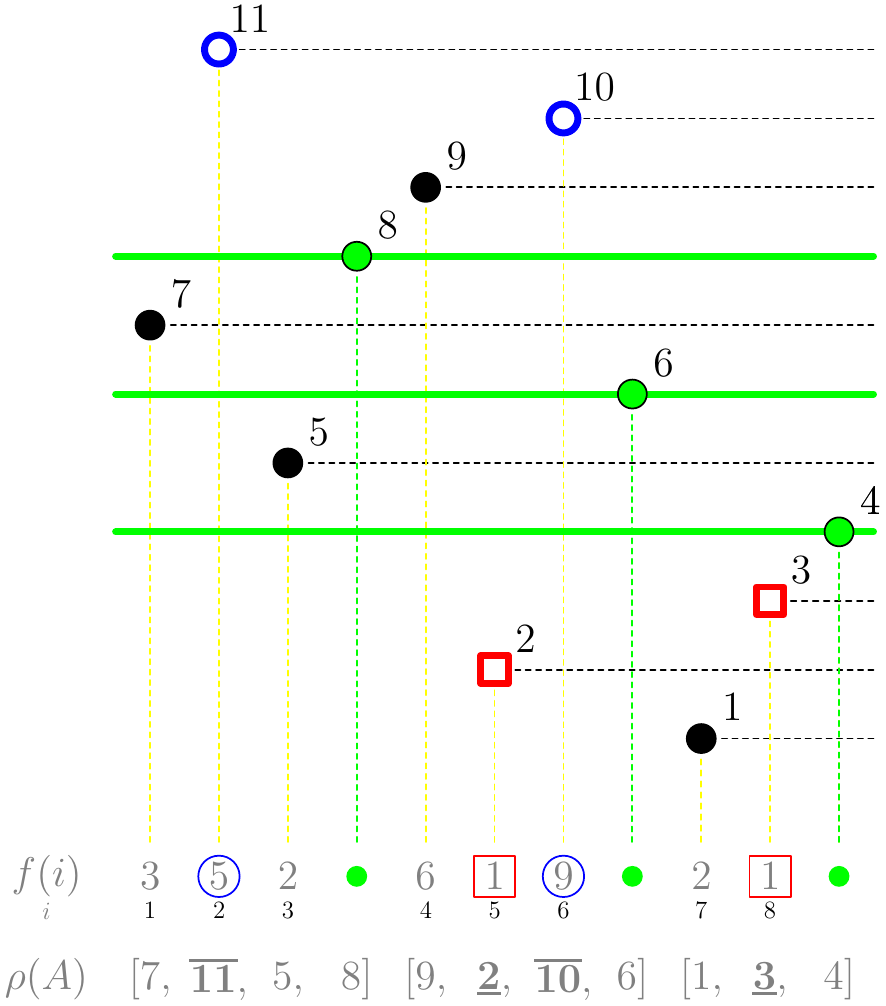}}
 \centering
  \caption{$\rho(A)$, with elements inserted with weights $\alpha^{-1}$ and $\beta^{-1}$ corresponding to the blue circles and red squares respectively, and also to the elements of $\lrs(\rho(A))$ (denoted by bar) and $\rls(\rho(A))$ (denoted by underline).}\label{transformation}
 \end{minipage}
\end{figure}

\begin{defn}
The total \emph{weight} of the final object is the product of the weights of all the inserted elements.
\end{defn}

For example, the weight of the assembl\'{e}e constructed in Figure \ref{insertion} is $\alpha^{-2}\beta^{-2}$.

Let $b_i$ be the relative height of the $i$'th green line from top to bottom for $1 \leq i \leq r+1$, and let $x_i$ be the relative height of the $i$'th inserted element for $1 \leq i \leq n-r$. An assembl\'{e}e $A$ is obtained by merging the sequence $[x_1,\ldots,x_{n-r}]$ with the sequence $[b_1,\ldots,b_{r+1}]$ so that the values $b_i$ are in the locations corresponding to the relative lateral placement of the green points in between the $x_i$'s, and they are set to be the block-ends of $A$ (See Figure \ref{insertion}). Now, this $A$ is \emph{almost} an assembl\'{e}e of size $(n+1,r+1)$ with the desired weight. To obtain the correct weight, we apply the transformation $\rho$ to $A$ guarantee that:
\begin{itemize}
\item The number of times an inserted element $x \in A$ was given weight $\alpha^{-1}$ equals $|\lrs(\rho(A))|$.
\item The number of times an inserted element $x \in A$ was given weight $\beta^{-1}$ equals $|\rls(\rho(A))|$.
\end{itemize}

\begin{defn}
Let $\mathbf{b}=[b_1,\ldots,b_{r+1}]$ be the sequence of block-end elements of $A$. Let the sequence $\lar(A) = [a_1,a_2,\ldots,a_u]$ consist of all the elements $a_i \in A$ such that $a_i>b_1$, where the order of the $a_i$'s matches the order of their appearance in $A$ from left to right. Similarly, let the sequence $\sma(A)=[c_1,c_2,\ldots,c_w]$ consist of all the elements $c_i$ such that $c_i<b_{r+1}$, where the order of the $c_i$'s matches the order of their appearance in $A$ from left to right. We define $\rho(A)$ to be the involution that replaces $a_i$ with $a_{u-i+1}$ for $1 \leq i \leq u$, replaces $c_i$ with $b_{r+1}-c_i$ for  $1 \leq i \leq w$, and leaves the rest of the entries of $A$ unchanged. We denote by $\rho(x)$ the element of $\rho(A)$, to which $x$ was sent, for $x \in A$
\end{defn}

\begin{remark} 
On $\lar(A)$, the transformation $\rho$ acts as the classical operation ``mirror image'' (in Figure \ref{transformation}, this means the points in $\lar(A)$ are reflected across a vertical axis), and on $\sma(A)$, $\rho$ acts as the classical operation ``complement'' (in Figure \ref{transformation}, this means the points in $\sma(A)$ are reflected across a horizontal axis). 
\end{remark}

\begin{example}
For $A = [7,10, 5, 8]\ [9, 2,11, 6]\ [3, 1, 4]$ with $\mathbf{b} = [8,6,4]$ from Figure \ref{insertion}, we have $\lar(A) = [10,9,11]$ and $\sma(A) = [2,3,1]$. In Figure \ref{transformation}, we obtain $\rho(\lar(A)) = [11, 9, 10]$ and $\rho(\sma(A)) = [2, 1, 3]$. Thus we obtain $\rho(A) = [7,11, 5, 8]\ [9, 2,10, 6]\ [1, 3, 4]$ with $\lrs(\rho(A)) = [11,10]$ and $\rls(\rho(A)) = [2,3]$, where the $\lrs$ and $\rls$ elements correspond precisely to those elements that were originally inserted with weight $\alpha^{-1}$ and $\beta^{-1}$ respectively.
\end{example}

\begin{thm}
Let $A$ be an assembl\'{e}e resulting from applying the insertion algorithm, and let $\wt_{\beta}$ and $\wt_{\alpha}$ be the number of times a non-block-end element $x \in A$ was given weight $\beta^{-1}$ and $\alpha^{-1}$ respectively. Then $|\rls(\rho(A))|=\wt_{\alpha}$ and $|\lrs(\rho(A))|=\wt_{\beta}$.
\end{thm}

\begin{proof}
We claim that $\rho$ implies the following characteristics.
\begin{enumerate}
\item[(i.)] If an element $x$ was inserted with weight $\alpha^{-1}$, then $\rho(x) \in \lrs(\rho(A))$. 
\item[(ii.)] If $x$ was inserted with weight $\beta^{-1}$, then $\rho(x) \in \rls(\rho(A))$.
\item[(iii.)] If $x$ was inserted with weight $1$, then $\rho(x) \not\in \lrs(A)$ and $\rho(x) \not\in\rls(\rho(A))$.
\end{enumerate}

(i.) If $x$ was inserted with weight $\alpha^{-1}$, then it is larger than $b_1$ and any element inserted before $x$, i.e. to its left. Thus when $\rho$ is applied to $A$, $\rho(x)$ is still larger than $b_1$, and also is larger than any element to its \emph{right} in $\rho(A)$. Consequently, $\rho(x) \in \lrs(\rho(A))$.

(ii.) If $x$ was inserted with weight $\beta^{-1}$, then it is smaller than $b_{r+1}$ and any element inserted before $x$, i.e. to its left. Thus when $\rho$ is applied to $A$, $\rho(x)$ is still smaller than $b_{r+1}$, and is now \emph{larger} than any element to its left in $\rho(A)$. Consequently, $\rho(x) \in \rls(\rho(A))$.

(iii.) If $x$ was inserted with weight $1$, then either $x \in \lar(A)$ and there is a larger element $x'$ to its left, or $x \in \sma(A)$ and there is a smaller element $x''$ to its left, or $b_1 > x > b_{r+1}$. In the first case, if $x \in \lar(A)$, then when $\rho$ is applied to $A$, we obtain $\rho(x')>\rho(x)$, with $\rho(x')$ to the right of $\rho(x)$, so $\rho(x) \not\in\lrs(\rho(A))$. In the second case, if $x \in \sma(A)$, then when $\rho$ is applied to $A$, we obtain $\rho(x'')>\rho(x)$, with $\rho(x'')$ to the left of $\rho(x)$, so $\rho(x) \not\in\rls(\rho(A))$. In the third case, by definition $x \not\in\lrs(A)$ and $x \not\in\rls(A)$.

The theorem follows.
\end{proof}

\section{Conclusion and further results}

Recall that in Section \ref{sec_intro}, we defined $\mathcal{Z}_{n,r}=\sum_{X \in B_n^r}\sum_{T \in R(\mathcal{T})}\wt(T)$. $\mathcal{Z}_{n,r}$ is also the partition function of the two-species ASEP with three parameters $\alpha$, $\beta$ and $q$ (i.e. the sum over the unnormalized steady state probabilities of all the states in $B_n^r$).

By combining the three bijections presented in Sections \ref{sec_bij} and \ref{sec_assemblee} of this paper, we directly get a bijective proof of the formula for the weight generating function for RAT with parameters $\alpha$ and $\beta$:
\[\mathcal{Z}_{n,r}(\alpha,\beta,q=1) = (\alpha\beta)^{n-r} {n \choose r} \prod_{i=r}^{n-1} \left( \alpha^{-1}+\beta^{-1}+i \right),\]
which is also the partition function of the two-species ASEP at $q=1$. 

It remains to find an interpretation of the parameter $q$ in terms of assembl\'{e}es. Furthermore, from the forest of binary trees of Section \ref{sec_reverse}, we can define an analog of the tree-like tableaux of \cite{treelike}. This tree-like analog on the RAT has vertices in each cell of the rhombic alternative tableau $T$ containing an $\alpha$ or $\beta$, and also at every edge on the north-west boundary of $T$ that is associated to a root edge, as in Figure \ref{reverse_init}. Such tableaux are in bijection with the RAT. From this paper, an analog of the insertion algorithm of \cite{treelike} should be defined for rhombic tree-like tableaux which would lead to a combinatorial interpretation of the parameter $q$ in terms of assembl\'{e}es of permutations.

\end{document}